\documentclass[12pt]{amsart}
\usepackage{amssymb}

\newtheorem{theorem}{Theorem}[section]

\newtheorem{lemma}{Lemma}[section]
\newtheorem{cor}{Corollary}[section]

\newcommand{\Q}{{\mathbb Q}}
\newcommand{\Z}{{\mathbb Z}}
\newcommand{\N}{{\mathbb N}}

\newcommand{\R}{{\mathbb R}}
\newcommand{\ml}{{\mathcal L}}

\title[Integers represented by Lucas sequences]{Integers represented by Lucas sequences}

\date{\today}

\author{Lajos Hajdu}
\address{Institute of Mathematics, University of Debrecen,\newline
\indent P. O. Box 400, H-4002 Debrecen, Hungary \newline
\indent and HUN-REN DE Equations, Functions, Curves and their Applications Research Group}
\email{hajdul@science.unideb.hu}

\author{Rob Tijdeman}
\address{Mathematical Institute, Leiden University,\newline
\indent Postbus 9512, 2300 RA Leiden, The Netherlands}
\email{tijdeman@ziggo.nl}

\thanks{Research supported in part by the HUN-REN Hungarian Research Network and by the NKFIH grant ANN 130909.}

\begin{document}

\begin{abstract}
In this paper we study the sets of integers which are $n$-th terms of Lucas sequences. We establish lower- and upper bounds for the size of these sets. These bounds are sharp for $n$ sufficiently large. We also develop bounds on the growth order of the terms of Lucas sequences that are independent of the parameters of the sequence, which is a new feature.
\end{abstract}

\maketitle

\section{Introduction}

In this paper we study sets of integers which are $n$-th terms of Lucas sequences for some $n\geq 0$. For integers $A,B$, the sequence $U=(U_n)_{n=0}^{\infty}$ with $U_0=0$, $U_1=1$ satisfying the binary recursive relation
$$
U_n=AU_{n-1}-BU_{n-2}\ \ \ (n\geq 2),
$$
is called a Lucas sequence. Note that the Fibonacci sequence is a Lucas sequence, corresponding to the choice $(A,B)=(1,-1)$. Lucas sequences are well studied in the literature, see for example \cite{bw} and the references there.

Write $f(x)=x^2-Ax+B$ for the characteristic polynomial of $U$, and let $\alpha,\beta$ be its roots. Throughout the paper, unless stated otherwise, we shall assume that the sequence is non-degenerate, that is, $AB\neq 0$ and $\alpha/\beta$ is not a root of unity. (It is easy to deal with the excluded cases; see Lemma \ref{lemdeg}.) Without loss of generality we assume that $|\alpha|\geq |\beta|$ throughout the paper. We have
\begin{equation}
\label{exprun}
U_n = \frac {\alpha^n - \beta^n}{\alpha - \beta}\ \ \ (n\geq 0).
\end{equation}
If $A^2>4B$ then $\alpha$ and $\beta$ are real and we say that we are in the real case. If $A^2<4B$, then $\alpha$ and $\beta$ are non-real complex numbers which are conjugates, and we say that we are in the non-real case.

We start with a new type bound for the growth of Lucas sequences. It has been long known by a classical result of Stewart \cite{st76} that these sequences grow exponentially, cf. Lemma 5 of \cite{ss}. However, the earlier bounds depend on $A$ and $B$. We show that $|U_n| \geq \frac 12 (\frac{1+\sqrt{5}}{2})^{n-2}$ in the real case and $|U_n|\geq (\sqrt{2})^{n-c(\log n)^2}$ in the non-real case for $n>1$, independently of the chosen Lucas sequence. Here $c$ is an absolute constant, which we give explicitly. The proof in the real case is elementary, while in the non-real case it is based on an estimate of linear forms in two logarithms.

Next we give sharp upper and lower bounds for the number of Lucas sequences with $|\alpha| \leq t$.

Thereafter for $n,N\in\N$ we study the sets
\begin{multline*}
\ml_n(N)=\{x\in\Z:0 \leq x\leq N\ \text{and}\ x=|U_n|\\ \text{for some non-degenerate Lucas sequence}\ U\}
\end{multline*}
and
$$
\ml_{\geq n}(N)=\bigcup_{m=n}^\infty \ml_m(N).
$$
We derive sharp, explicit upper- and lower bounds for the growth order of $|\ml_n(N)|$ and $|\ml_{\geq n}(N)|$. 
In the proofs we need to combine several tools, including our new bounds on the growth of Lucas sequences, extensions of theorems of Erd\H{o}s and Mahler \cite{em} and Lewis and Mahler \cite{lm} concerning representability of integers by binary forms $G(x,y)$ to representations of the type $G(x^2,y)=k$, and certain properties of the Fibonacci polynomials. For the introduction of these polynomials, observe that the first few terms of $U$ are given by
$$
0,\ 1,\ A,\ A^2-B,\ A^3-2AB,\ A^4-3A^2B+B^2,\ A^5-4A^3B+3AB^2. 
$$
We define $F_n(x,y)$ by $F_0(x,y)=0$, $F_1(x,y)=1$ and
$$
F_n(x,y) = xF_{n-1}(x,y) - yF_{n-2}(x,y)\ \ \ (n\geq 2).
$$
Polynomials $F_n(A,B)$ correspond with Lucas sequences, polynomials $F_n(x,-1)$ are the Fibonacci polynomials, polynomials $F_n(2x,1)$ are the Chebyshev polynomials of the second kind.

The structure of the paper is the following. In Section \ref{mres} we formulate our principal results. In Section \ref{groluc} we prove Theorem \ref{thm21} which provides lower bounds on the growth of Lucas-sequences, in Section \ref{numLuc} we prove Theorem \ref{lemup1} which gives precise bounds for the number of Lucas sequences with bounded $|\alpha|$, in Section \ref{lucnum} we prove Theorem \ref{thm22} on upper bounds for the sizes of the sets $\ml_n$ and $\ml_{\geq n}$, and in Section \ref{lower} we prove Theorems \ref{thm23} and \ref{thm24} on lower bounds for them. These upper and lower bounds differ by a multiplicative constant depending only on $n$ for $n \geq 7$ odd, and by a lower order factor for $n=5$ and for $n\geq 6$ even. 

By $c_1, c_2, c_3, \dots$ we denote effectively computable constants depending only on $n$.


\section{Main results}
\label{mres}

Lemma 5 of \cite{ss} (taken from \cite{st76}) states that there exist positive constants $N_0$ and $C_0$ depending on $A,B$ such that for $n>N_0$ we have $|U_n|>|\alpha|^{n-C_0\log n}$. Our first theorem in combination with Remark 1 and Remark 2 implies that $N_0$ and $C_0$ can be chosen independently of $A$ and $B$.

\begin{theorem}
\label{thm21}
For $n\geq 2$ we have in the real case
$$
|U_n| \geq \frac{|\alpha|^{n-2}}{2}\ \ \ \text{if}\ B<0
$$
$$
|U_n| \geq |\alpha|^{n-1} \ \ \text{if}\ A^2>4B>0
$$
and in the non-real case
$$
|U_n|\geq
\begin{cases}
\frac{1}{4}e^{-250(\log n)^2}|\alpha|^{n-2},&\text{if}\ n>5\cdot 10^{8},\\
\frac{1}{4}e^{-100000}|\alpha|^{n-2},&\text{if}\ n\leq 5\cdot 10^{8}
\end{cases}
$$
for $B\leq 535$, and
$$
|U_n|\geq
\begin{cases}
\frac{1}{4}|\alpha|^{n-2-88(\log n)^2},&\text{if}\ n>2.1\cdot 10^{8},\\
\frac{1}{4}|\alpha|^{n-31710},&\text{if}\ n\leq 2.1\cdot 10^{8}
\end{cases}
$$
for $B\geq 536$.
\end{theorem}

\noindent {\bf Remark 1.}  We have in the real case
$$
|\alpha| = \frac{1+\sqrt{5}}{2} \ \ \text{if} \ (A,B) = \pm (1,-1), \ \ |\alpha| \geq 2\ \ \ \text{if} \ (A,B)\neq \pm (1,-1),
$$
and in the non-real case, since $B=1$ is excluded,
$$
|\alpha| = \sqrt{B} \geq \sqrt{2}.
$$ 

\noindent{\bf Remark 2.} In the proof of Theorem \ref{thm21} in the non-real case we use a result of Laurent \cite{la} on linear forms in two logarithms. This causes the appearance of $(\log n)^2$ in the exponent, but with a rather small coefficient. If we use the Baker-type result of Matveev \cite{ma} instead, then the $(\log n)^2$ is replaced by $\log n$, however, with much larger constants as coefficients. The following argument, that we owe to an unknown colleague, shows that it is not true that there is an absolute constant $c>0$ such that $|U_n| > c |\alpha|^n$ for all $n$ in the non-real case. Write $\frac {\beta}{\alpha}=e^{2 \pi i \theta}$. Let $\frac mn$ be a convergent of $\theta$. Then  $|n \theta - m| < \frac 1n$. Hence $\left| \left(\frac{\beta}{\alpha}\right)^n-1 \right| \ll \frac 1n$. Thus $|\alpha^n - \beta^n| \ll \frac{|\alpha|^n}{n}$ for infinitely many $n$.

\vskip.2cm

The following result provides an upper and a lower bound for the number of Lucas sequences with bounded roots. It shows that this number is $4t^3 + O(t^2)$.
\begin{theorem}
\label{lemup1}
Let $t \in {\R}_{\geq 2}$. Then the number of non-degenerate Lucas sequences with $|\alpha| \leq t$ is at most $4t^3-t^2+7t$ and at least $4t^3-10t^2-29t$.
\end{theorem}

\noindent{\bf Remark 3.} The upper- and lower bounds in Theorem \ref{lemup1} can slightly be improved if $t$ is an integer.

\vskip.2cm

Observe the following facts. Here the degenerate cases are too few to influence the density (cf. Lemma \ref{lemdeg}).
\begin{itemize}
\item For every non-zero integer $K$ there exist infinitely many $U$ with $U_2=K$. For this we may choose $A=K$ with $B$ arbitrary.
\item For every integer $K$ there exist infinitely many $U$ with $U_3=K$. For this we may choose $B=A^2-K$ with $A$ arbitrary.
\item Let $U_4=K$. 
Observe that $A$ and $A^2-2B$ are both odd or both even, hence $K\not\equiv 2\pmod{4}$.
If $K$ is odd, we can take $A=1$ and $B=\frac{1-K}{2}$. If $4$ divides $K$, we can take $A=2$ and $B=2-\frac{1}{4}K$.
\end{itemize}
We conclude that  $$\lim_{N \to \infty}\frac{\ml_n(N)}{N} =1,1,\frac{3}{4}, \ {\rm for} \ n=2,3,4, {\rm respectively}.$$ In the sequel we restrict our attention to $n\geq 5$.

The following statement provides an upper bound for the number of $n$-th terms up to $N$ counted over all non-degenerate Lucas sequences.

\begin{theorem}
\label{thm22}
{\rm i)}  $|\ml_n(N)|\leq \left(612 + 2^{20-n} n^2\right) N^{\frac {3}{n-1}}$ for $n \geq 5$.  \\
{\rm ii)} $|\ml_{\geq n}(N)| \leq \left(612 +2^{22-n}n^2\right)N^{\frac {3}{n-1}}+ 1836N^{\frac {3}{n} }\log N$ for $n \geq 5$.
\end{theorem}

\noindent Theorem \ref{thm22} implies that $\ml_n(N) \ll N^{\frac {3}{n-1}}$ and the density of $\ml_{\geq 5}$ is zero. In fact, since $N^{\frac 35} \log N < 6 N^{\frac 34}$ for all $N\geq 1$,
\begin{cor}
\label{thm25}
$$|\ml_n(N)| \leq (612+o_n(1))N^{\frac{3}{n-1}}\ {\rm for ~all~}n,\ {\rm and}\ |\ml_{\geq 5}(N)| \leq 3.3 \cdot10^6 \cdot N^{\frac 34}.$$
\end{cor}

The next statements show that the exponents $\frac {3}{n-1}$ in Theorem \ref{thm22} are best possible. 
\begin{theorem} \label{thm23}
Let $n\geq 7$ and odd. Then there exists a positive constant $c_1$ depending only on $n$ such that
$$
|\ml_n(N)| \geq c_1N^{\frac{3}{n-1}}, \ \ {\rm hence} \ \  |\ml_{\geq n}(N)| \geq c_1N^{\frac{3}{n-1}}.
$$
provided that $|\ml_n(N)|$ and $|\ml_{\geq n}(N)|$, respectively, are positive.
\end{theorem}
\noindent Thus for odd values $n\geq 7$ the bounds are sharp apart from a multiplicative constant depending only on $n$.

\begin{cor}
For $n \geq 7, n$ odd we have
$$ |\ml_n(N)| \ll\gg_n  N^{\frac{3}{n-1}}.$$
\end{cor}
\noindent For other values of $n \geq 5$ our lower bound is slightly weaker, but of the same order of magnitude.

\begin{theorem} \label{thm24}
Let $n\geq 5$. Then there exists a positive constant $c_2$ depending only on $n$ such that
$$
|\ml_n(N)| \geq c_2N^{\frac{3}{n-1}- \frac{4}{\log \log N}}, \ \ {\rm hence} \ \  |\ml_{\geq n}(N)| \geq c_2N^{\frac{3}{n-1}- \frac{4}{\log \log N}},
$$
provided that $|\ml_n(N)|$ and $|\ml_{\geq n}(N)|$, respectively, are positive.
\end{theorem}

\noindent {\bf Remark 4.} We expect that, for all $n \geq 5$, 
$$
\ml_n(N) \gg_n N^{\frac {3}{n-1}} \ \ {\rm and \  therefore} \ \ \ml_{\geq n}(N) \gg_n N^{\frac {3}{n-1}}.
$$

Because of the symmetry of the polynomials $x^2-Ax+B =\\(x- \alpha)(x-\beta)$ and $x^2+Ax +B=(x+\alpha)(x+\beta)$ we assume that $A>0$ in the sequel, unless stated otherwise. We further assume that in the real case $|\alpha| > |\beta|$. It follows that $\alpha > 0$.

\section{Proof of Theorem \ref{thm21}}
\label{groluc}

To prove Theorem \ref{thm21} in the case where $\alpha,\beta$ are non-real, we shall use bounds for linear forms in two logarithms. For this, we introduce some notation. For an algebraic number $\gamma$ of degree $d$ over $\Q$, the absolute logarithmic height of $\gamma$ is defined by
$$
h(\gamma)=\frac{1}{d}\left(\log|a|+\sum\limits_{i=1}^d \log\max(1,|\gamma^{(i)}|)\right)
$$
where $a$ is the leading coefficient of the minimal polynomial of $\gamma$ over $\Z$, and the $\gamma^{(i)}$'s are the algebraic conjugates of $\gamma$. Let $\alpha_1$, $\alpha_2$ be non-zero algebraic numbers. Consider the linear form
$$
\Lambda=b_2\log\alpha_2 - b_1\log \alpha_1
$$
where $b_1,b_2$ are non-zero integers, and the logarithm of a non-zero complex number $w$ (here and later on) is taken according to
$$
\log w = \log |w| + {\mathrm i}\arg w
$$
with $-\pi < \arg w \leq \pi$. Set $$D=\frac{[\Q(\alpha_1,\alpha_2):\Q]}{[\R(\alpha_1,\alpha_2):\R]} \ \ { \rm and } \ \
b'= \frac{b_1}{D\log A_2}+\frac{b_2}{D\log A_1},
$$
where $A_1,A_2$ are real numbers greater than $1$ such that
$$
\log A_i \geq \max\left\{h(\alpha_i),\ \frac{|\log \alpha_i|}{D},\ \frac 1D \right\}\ \ \ (i=1,2).
$$

The following result is due to Laurent, see Corollary 1 in \cite{la}.

\begin{lemma}
\label{lemla}
Keeping the above notation, suppose that $\alpha_1,\alpha_2$ are multiplicatively independent. Then we have
$$
\log |\Lambda|\geq -25.2D^4\left(\max\left\{\log b' + 0.21,\ \frac{20}{D},\ 1\right\}\right)^2\log A_1\log A_2.
$$
\end{lemma}

The following consequence of Lemma \ref{lemla} will be useful.

\begin{lemma}
\label{lemb}
Let $\delta$ be an algebraic number which is not a root of unity. Then for any non-zero integer $\ell$ we have
$$
\left|\delta^\ell-1\right|\geq\frac{1}{2}\exp\left(-25.2\pi D_\delta^3\left(\max\left\{\log\frac{(\pi+D_\delta\log A_\delta)|\ell|}{\pi D_\delta\log A_\delta} + 0.21,\frac{20}{D_\delta},1\right\}\right)^2\log A_\delta\right),
$$
where
$$
D_\delta=\frac{[\Q(\delta):\Q]}{[\R(\delta):\R]}
$$
and $A_\delta$ is a real number with
$$
\log A_\delta\geq \max \left\{h(\delta),\ \frac{|\log \delta|}{D_\delta},\ \frac{1}{D_\delta} \right\}.
$$
\end{lemma}

\begin{proof}
Put $z=\delta^\ell-1$. Since $\delta$ is not a root of unity, we have $z\neq 0$. If $|z|>\frac 12$, then our claim immediately follows. So we may assume that $|z|\leq \frac 12$. Then, as it is well-known, we have
\begin{equation}
\label{ez}
|\log(1+z)|\leq 2|z|.
\end{equation}
On the other hand, we also have
\begin{equation}
\label{az}
\log(1+z)=\ell\log(\delta)+2k\log(-1)
\end{equation}
with some integer $k$. As $|\log(1 + z)|\leq 1$, we have $|\Im(\log(1+z))|\leq 1$. Hence, by the choice of the logarithm,
$$
|\ell\Im(\log(\delta))+2k\pi |\leq 1
$$
which implies
$$
|2k|\leq \frac{1+\pi|\ell|}{\pi}=|\ell|+\frac{1}{\pi}.
$$
As $2k$ and $\ell$ are integers, this in fact gives
$$
|2k|\leq|\ell|.
$$
Hence, in view of \eqref{ez} and \eqref{az}, our claim follows from Lemma \ref{lemla}.
\end{proof}

\vskip.2cm

\noindent
{\bf Proof of Theorem \ref{thm21}.}
Let $n \geq 2$.\\
{\it Real case.} Suppose that $A^2>4B$. Then $\alpha$ and $\beta$ are real. As stated at the end of Section \ref{mres} we assume $A\geq 1$ and $\alpha>|\beta|$. Put $D=\sqrt{A^2-4B}$.

Assume first that $B<0$. If $(A,B)=(1,-1)$, then $\alpha=\frac{1 +\sqrt{5}}{2}$, $\beta=\frac{1 - \sqrt{5}}{2}$ and
$
U_n=\frac{\alpha^n-\beta^n}{\sqrt{5}}.
$
Since $|\beta| < 1$, we have
$$ U_n > \frac{\alpha^n - 1}{\sqrt{5}} > \frac{\alpha^{n-2}}{2}.$$ 
In all other cases we have $\alpha \geq 2$. Then we infer, by $\beta = \frac{B}{\alpha} < 0, | \beta| < \alpha$,
\begin{equation}
\label{bsm}
U_n=\frac{\alpha^n-\beta^n}{\alpha-\beta} \geq\frac{\alpha^n-|\beta|^n}{\alpha-|\beta|} \cdot \frac {\alpha-|\beta|}{\alpha + |\beta|} \geq \alpha^{n-1} \frac{A}{2\alpha} \geq \frac{\alpha^{n-2}}{2}.
\end{equation}

Assume next that $B>0$. Then $A\geq 3$, $\beta > 0$ and, on using that $D \geq 1$,
\begin{equation}
\label{Unest}
U_n=\frac{\alpha^n-\beta^n}{\alpha-\beta}\geq\alpha^{n-1}.
\end{equation}

\noindent {\it Non-real case.} Suppose now $A^2-4B<0$. Then $B>0$ and $\alpha$ and $\beta$ are complex conjugates. If $B=1$, then $A= 1$ and $\frac{\alpha}{ \beta}$ is a root of unity, which is excluded. So we may assume that $B>1$. Then we have 
\begin{equation}
\label{negc}
|U_n|=\left|\frac{\alpha^n-\beta^n}{\alpha-\beta}\right|=\left|\frac{\alpha^n}{\alpha-\beta}\right|\left|\left(\frac{\beta}{\alpha}\right)^n-1\right|.
\end{equation}

We shall use that $\frac{\beta}{\alpha}$ is a root of the polynomial
$$
B\left(x-\frac{\alpha}{\beta}\right)\left(x-\frac{\beta}{\alpha}\right)=Bx^2-(A^2-2B)x+B
$$
which is irreducible over $\Q$. Thus, on using that $|\alpha|=|\beta|$, 
\begin{equation}
\label{betairr2}
h\left(\frac{\beta}{\alpha}\right)=\frac{\log B}{2}.
\end{equation}
Note that
$$
\left|\log\frac{\beta}{\alpha}\right|\leq \left|\log\left|\frac{\beta}{\alpha}\right|\right|+\pi=\pi.
$$
We distinguish two cases.

Assume first that $B<e^{2\pi}<536$. Then we apply Lemma \ref{lemb} for $\delta=\frac{\beta}{\alpha}$ with $D_\delta=1$, $\ell=n$, $A_\delta=e^\pi$.
We get
$$
\left|\left(\frac{\beta}{\alpha}\right)^n-1\right|\geq
\frac{1}{2}\exp\left(-25.2\pi^2\left(\max\left\{\log\frac{2n}{\pi} + 0.21,20\right\}\right)^2\right).
$$
This easily yields
$$
\left|\left(\frac{\beta}{\alpha}\right)^n-1\right|\geq
\frac{1}{2}\exp\left(-25.2\pi^2\left(\max\left\{\log n - 0.24,20\right\}\right)^2\right).
$$
From this with some calculations we get that
\begin{equation}
\label{baku0}
\left|\left(\frac{\beta}{\alpha}\right)^n-1\right|\geq
\begin{cases}
\frac{1}{2}e^{-250(\log n)^2},&\text{if}\ n>5\cdot 10^{8},\\
\frac{1}{2}e^{-100000},&\text{if}\ n\leq 5\cdot 10^{8}.
\end{cases}
\end{equation}
Combining it with \eqref{negc} we conclude that
$$
|U_n|\geq
\begin{cases}
\frac{1}{4}e^{-250(\log n)^2}B^{n/2-1},&\text{if}\ n>5\cdot 10^{8},\\
\frac{1}{4}e^{-100000}B^{n/2-1},&\text{if}\ n\leq 5\cdot 10^{8}.
\end{cases}
$$

Suppose next that $B\geq e^{2\pi}>535$. Now we apply Lemma \ref{lemb} for $\delta=\frac{\beta}{\alpha}$ with $D_\delta=1$, $\ell=n$, $A_\delta=\sqrt{B}$ to obtain
$$
\left|\left(\frac{\beta}{\alpha}\right)^n-1\right|\geq
\frac{1}{2}\exp\left(-12.6\pi(\log B)\left(\max\left\{\log\frac{(2\pi+\log B)n}{\pi\log B} + 0.21,20\right\}\right)^2\right).
$$
This by $B\geq 536$ implies
$$
\left|\left(\frac{\beta}{\alpha}\right)^n-1\right|\geq
\frac{1}{2}\exp\left(-12.6\pi(\log B)\left(\max\left\{\log n + 0.85,20\right\}\right)^2\right).
$$
Hence we get
\begin{equation}
\label{baku}
\left|\left(\frac{\beta}{\alpha}\right)^n-1\right|\geq
\begin{cases}
\frac{1}{2}B^{-44(\log n)^2},&\text{if}\ n>2.1\cdot 10^{8},\\
\frac{1}{2}B^{-15854},&\text{if}\ n\leq 2.1\cdot 10^{8}.
\end{cases}
\end{equation}
Thus from \eqref{negc} we obtain
$$
|U_n|\geq
\begin{cases}
\frac{1}{4}B^{\frac n2 -1-44(\log n)^2},&\text{if}\ n>2.1\cdot 10^{8},\\
\frac{1}{4}B^{\frac n2 -15855},&\text{if}\ n\leq 2.1\cdot 10^{8},
\end{cases}
$$
and our claim follows.
\qed


\section{Bounds for the number of Lucas sequences with bounded roots}
\label{numLuc}

In this section we prove Theorem \ref{lemup1}. For this (and also later on), the following lemma describing degenerate Lucas sequences will be useful. In fact, this description is simple and well-known, however, for the convenience of the reader we provide its proof, as well.

\begin{lemma}
\label{lemdeg}
Let $A,B$ be integers with $AB(A^2-4B)\neq 0$. Then the Lucas sequence $U_n$ belonging to $(A,B)$ is degenerate precisely for
\begin{equation}
\label{eqdeg}
(A,B)=(r,r^2),\ (2r,2r^2),\ (3r,3r^2)\ \ \ (r\in \Z\setminus\{0\}).
\end{equation}
\end{lemma}

\begin{proof}
As before, let $\alpha,\beta$ be the roots of $x^2-Ax+B$. As $B\neq 0$, $\alpha\beta\neq 0$. If $\varepsilon:=\frac{\alpha}{\beta}$ is a root of unity, then as $A(A^2-4B)\neq 0$, we see that $\alpha,\beta$ are non-real. So $\varepsilon$ is a quadratic non-real algebraic number which is a root of unity, implying
$$
\varepsilon\in\left\{\pm i,\pm\frac{1}{2}\pm\frac{i\sqrt{3}}{2}\right\}.
$$
In view of $\varepsilon\beta^2=B$ and $(1+\varepsilon)\beta=A$, we get that
$$
\beta=\frac BA \cdot \frac{1+\varepsilon}{\varepsilon} = r(1+\varepsilon^{-1}), \ \  \alpha = \frac {B}{\beta} =\varepsilon \beta = r(\varepsilon + 1),
$$
with $r= \frac BA$. Hence
$$A = \alpha + \beta = r(\varepsilon + 2 + \varepsilon^{-1}) \in \{2r,r,3r\}, \  B \in \{2r^2, r^2, 3r^2\}.$$
Since $B$ is a non-zero integer, we have $r\in\Z\setminus\{0\}$. Thus $(A,B)$ is given by \eqref{eqdeg}.

On the other hand, one can easily check for $(A,B)$ as in \eqref{eqdeg} that $\frac{\alpha}{\beta}$ is a root of unity, and the claim follows.
\end{proof}

\begin{proof}[Proof of Theorem \ref{lemup1}]
First we derive an upper bound for the number of Lucas sequences with $|\alpha|\leq t$, where $t\in\R_{\geq 2}$. Here, instead of assuming that $U_n$ is non-degenerate, for simplicity we shall only use that $AB(A^2-4B)\neq 0$. Again we assume $A>0$, and $\alpha > 0$ in the real case, and double the number of counted polynomials at the end.

Assume that $\alpha,\beta$ are non-real complex numbers. Then we have $A^2-4B<0$, hence
$$
0 < A < 2 \sqrt{B},  \ \ 0 < B = | \alpha|^2 \leq t^2.
$$
Thus we obtain that the number of possible pairs $(A,B)$ is at most 
$$ \sum_{1\leq B\leq t^2} 2 \sqrt{B} \leq 2 \int_0^{t^2} \sqrt{x} ~dx + 2t= \frac 43 t^3 +2t.$$

Suppose next that $\alpha,\beta$ are real, so $A^2 >4B$. Note that we have
$$
\alpha=\frac{A+\sqrt{A^2-4B}}{2}\leq t.
$$
We distinguish two subcases. If $B>0$, then $\sqrt{A^2 -4B} \leq 2t-A$, hence 
$$0 < A < 2t, \ \  \max(0, At-t^2)  \leq B< \frac 14 A^2.$$
In this subcase the number of pairs $(A,B)$ is at most
$$
\sum_{1\leq A\leq t} \left(\frac {1}{4}A^2\right) +\sum_{t<A < 2t} \left(\frac {1}{4}A^2-At+t^2+1\right)<\frac 16 t^3+\frac 43 t.
$$
(Note that here, and later on in similar situations in the proof, in view of that $t$ is not necessarly integer, a careful calculation is needed.) 

On the other hand, if $B<0$, then $A< t$ and $A + \sqrt{A^2-4B} < 2t$, hence $A^2-4B <  4t^2 - 4At + A^2$. From
$$
 0 < A < t, \ \ 0 < -B\leq t^2-At,
$$
it follows that the number of such possible pairs $(A,B)$ in this subcase is at most 
$$
\sum_{1\leq A< t} (t^2 - At) \leq \frac 12t^3- \frac 12 t^2 + \frac 18 t.
$$
Adding up the three bounds for the possible pairs $(A,B)$ and doubling this number, we obtain that the number of Lucas sequences with $AB(A^2-4B) \neq 0$ is at most $4t^3-t^2+7t$.

Now we give a lower bound for the number of non-degenerate Lucas sequences with $|\alpha| \leq t$ and $A>0$, whence $\alpha>0$ in the real case. We apply a similar case-by-case analysis as for the upper bound, however, now we need to take care of the degenerate Lucas sequences. These sequences are completely characterized by Lemma \ref{lemdeg}. Observe that for fixed $A> 0$, there are at most three $B$-s with $B(A^2-4B)\neq 0$ such that $(A,B)$ generates a degenerate Lucas sequence, and for fixed $B\neq 0$, there are at most three $A$-s with $A(A^2-4B)\neq 0$ such that $(A,B)$ generates a degenerate Lucas sequence. Note that different pairs $(A,B)$ yield different Lucas sequences.

The non-real case. Choose $0<B\leq t^2$, $0< A < 2 \sqrt{B}$. Then $AB\neq 0$, $A^2-4B<0$ and $|\alpha|=\sqrt{B} \leq t$. The number of such non-degenerate sequences is at least 
$$ \sum_{1\leq B\leq t^2-1} (2 \sqrt{B} -1) -3t^2\geq 2 \int_{0}^{t^2} \sqrt{x}~dx -2t -4t^2 = \frac 43 t^3 - 4t^2 -2t.$$

Next the real case with $B>0$. Choose $0 < A \leq t$, $0 < B < \frac 14 A^2$ and $t < A \leq 2t$, $At-t^2 \leq  B < \frac 14 A^2$. Then $AB\neq 0$, $A^2-4B>0$ and
$$
\alpha = \frac{A+\sqrt{A^2-4B}}{2} \leq t.
$$
The number of such non-degenerate sequences is at least 
$$
\sum_{1\leq A\leq t} \left(\frac{A^2}{4} - 1\right)+  \sum_{t<A\leq 2t} \left(\frac{A^2}{4}-At+t^2-1\right) -6t > \frac 16 t^3-\frac 12 t^2 - \frac{17}{2}t.
$$

Finally, the real case with $B<0$. Choose $0 < A \leq t$, $0<-B <  t^2 - At$. Then $AB\neq 0$, $A^2-4B>0$, and
$$
\alpha=\frac{A+\sqrt{A^2-4B}}{2} < \frac 12 (A + 2t-A) = t.
$$
The number of such non-degenerate sequences is greater than 
$$
 \sum_{1\leq A\leq t} (t^2 -At-1)-3t \geq \frac 12 t^3-\frac 12 t^2-4t.
$$

So we have altogether more than $4t^3 - 10t^2 -29t$ non-degenerate Lucas sequences with $|\alpha| \leq t$.

\end{proof}


\section{Proof of the upper bounds in Theorem \ref{thm22}}
\label{lucnum}

Suppose $|U_n| \leq N$ for the Lucas pair $(A,B)$ and index $n$. Again we assume $A>0$, hence $\alpha >0$ in the real case, and double the number of polynomials at the end. We include the degenerate cases with \\$AB(A^2-4B) \neq 0$ in the upper bound. We distinguish four cases.
\vskip.1cm
\noindent {\bf Case 1.} $0<A < 9N^{\frac{1}{n-1}}$ and $|B| < 17N^{\frac{2}{n-1}}$. Then the number of pairs $(A,B)$ with $|F_n(A,B)| \leq N$ is at most $306N^{\frac{3}{n-1}}$. \\
 \vskip.1cm
\noindent {\bf Case 2.} $A \geq 9N^{\frac{1}{n-1}}$, $|B| < 17N^{\frac{2}{n-1}}$. Thus $A^2>4B$, thus we are in the real case. By \cite{hb} we know the zeros of the Fibonacci polynomial $F_n(x,-1)$. So it follows that the roots of $F_n(x,1)$ are given by 
\begin{equation} \label{fibpol}
2\cos\frac{k\pi}{n} \ \ (k=1,\dots,n-1).
\end{equation}
Thus we have
$$F_n(A,B) = \prod_{k=1}^{\frac{n-1}{2}} \left(A^2 - 4B \cos^2 \frac{k\pi}{n}\right)$$
if $n$ is odd, and 
$$F_n(A,B) = A\prod_{k=1}^{\frac{n-2}{2}} \left(A^2 - 4B \cos^2 \frac{k\pi}{n}\right)$$
if $n$ is even. Since $A \geq 9N^{\frac{1}{n-1}}$ and 
$A^2 - 4B \cos^2 \frac{k\pi}{n} \geq A^2 - 4|B| > N^{\frac{2}{n-1}},$
we obtain 
$$N \geq |F_n(A,B)| > \min \left((N^{\frac{2}{n-1}})^{\frac{n-1}{2}},9N^{\frac{1}{n-1}}(N^{\frac{2}{n-1}})^{\frac{n-2}{2}}\right) =N.$$
Thus there are no pairs $(A,B)$ with $|F_n(A,B)| \leq N$ in this case. 
\vskip.1cm
\noindent {\bf Case 3.} $B \leq -17N^{\frac{2}{n-1}}$. We have, in view of \eqref{fibpol},
\noindent 
$$
F_n(A,B) = \prod_{k=1}^{n-1} \left(A - 2\sqrt{B} \cos \frac{k\pi}{n}\right).
$$
As $\sqrt{B}$ is purely imaginary, $|A - 2\sqrt{B} \cos \frac{k\pi}{n}| \geq 2\sqrt{|B|} \  |\cos \frac{k\pi}{n}|$. Thus, by $|\cos \frac{k\pi}{n}| = |\cos \frac{(n-k)\pi}{n}|$,$$N \geq |F_n(A,B)| \geq \left(\prod_{k=1}^{\lfloor \frac{n-1}{2}\rfloor} (2 \sqrt{|B|}\cos \frac{k\pi}{n}\right)^2.$$

If $n$ is odd, then we get, on using that $\cos x \geq 1 - \frac {2}{\pi}x$ for $0 \leq x \leq \frac{\pi}{2}$,
$$\prod_{k=1}^{\frac{n-1}{2}} \cos \frac{k\pi}{n} \geq \prod_{k=1}^{\frac{n-1}{2}}\left(1- \frac{2k}{n}\right)= \frac {(n-1)!}{(\frac{n-1}{2})! (2n)^{\frac{n-1}{2}}}>\left(\frac {n+1}{4n}\right)^{\frac{n-1}{2}}>2^{-n+1}.$$
Thus, by $B \leq -17N^{\frac{2}{n-1}}$,
\begin{equation} \label{oddcas}
N \geq (2\sqrt{|B|})^{n-1} 4^{-n+1} \geq \left(\frac 12 \sqrt{17}N^{\frac{1}{n-1}}\right)^{n-1} > N,
\end{equation}
a contradiction. So there are no such pairs $(A,B)$.

If $n$ is even, then we get similarly, by $n \geq 6$,
$$\prod_{k=1}^{\frac{n-2}{2}} \cos \frac{k\pi}{n} \geq \prod_{k=1}^{\frac{n-2}{2}}\left(1- \frac{2k}{n}\right) =\frac {2^{\frac{n-2}{2}} \left(\frac{n-2}{2}\right)!}{n^{\frac{n-2}{2}}}> \left( \frac 2n \cdot \frac{n-2}{2e}\right)^{\frac{n-2}{2}} >2^{-n+2}.$$
As in the case $n$ odd we find that there are no such pairs $(A,B)$.
\vskip.1cm
\noindent {\bf Case 4.} $B \geq 17N^{\frac{2}{n-1}}$.
We use the formula
$$\cos x - \cos y  = -2 \sin \frac {x+y}{2} \sin \frac{x-y} {2}.$$
Let $k_0$ be the value of $k$ such that $ \cos \frac{k_0\pi}{n}$ is nearest to $\frac{A}{2 \sqrt{B}}$. Then
$$
\left|\underset{k\neq k_0}{\prod_{k=1}^{n-1}} \left(A - 2\sqrt{B} \cos \frac{k\pi}{n}\right)\right| \geq $$
$$(2\sqrt{B})^{n-2} \left| \left(\cos \frac {(k_0 -0.5)\pi}{n} - \cos\frac {(k_0 - 1)\pi}{n} \right) \left( \cos \frac {(k_0 +0.5)\pi}{n} - \cos \frac {(k_0 + 1)\pi}{n} \right) \right| $$
$$\times \left| \prod_{k=1}^{k_0-2} \left( \cos \frac {(k_0 - 1)\pi}{n} - \cos\frac {k\pi}{n} \right) \prod_{k=k_0+2}^{n-1} \left(\cos\frac {(k_0 + 1)\pi}{n} - \cos \frac {k\pi}{n} \right) \right| =$$
$$ (4 \sqrt{B})^{n-2}  \left| \sin \frac{(2k_0-1.5) \pi}{2n} \cdot \sin \frac {\pi}{4n} \cdot \sin \frac {(2k_0+1.5) \pi}{2n}  \cdot \sin \frac{\pi}{4n} \right|$$
$$\times\left|\prod_{k=1}^{k_0-2}  \left( \sin \frac {(k+k_0-1)\pi}{2n}  \cdot \sin \frac{(k_0-k-1)\pi}{2n} \right)\right|$$
$$\times\left|\prod_{k=k_0+2}^{n-1}  \left( \sin \frac {(k+k_0+1)\pi}{2n} \cdot \sin \frac{(k-k_0-1)\pi}{2n} \right)\right|. $$
On using that $ \sin x \geq \frac {2}{\pi} x$ for $0 \leq x \leq \pi$, we obtain that this is larger than
$$(4 \sqrt{B})^{n-2} \left|  \frac{2(k_0-1)}{n} \cdot \frac {1}{2n} \cdot \frac {2k_0}{n}  \cdot \frac{1}{2n} \right|$$
$$\times\prod_{k=1}^{k_0-2}  \left(  \frac {k+k_0-1}{n}  \cdot  \frac{k_0-k-1}{n} \right) \times \prod_{k=k_0+2}^{n-1}  \left(  \frac {k+k_0+1}{n} \cdot  \frac{k-k_0-1}{n} \right)= $$
$$  \left(\frac{4 \sqrt{B}}{n^2} \right)^{n-2}  (k_0-1)k_0 \frac{(2k_0-3)!}{(k_0-1)!} (k_0-2)! \frac{(n+k_0)!}{(2k_0+2)!} (n-k_0-2)!\geq $$
$$ \frac{1}{32n^4} \left( \frac {4\sqrt{B}}{n^2} \right)^{n-2} (n+k_0)! (n-k_0-2)!.$$
It is not difficult to check that 
$$m! \geq \left( \frac me\right)^m, \ \  \left(1+ \frac mn\right)^{n+m} > e^m, \ \  \left(1 - \frac mn\right)^{n-m} > e^{-m}$$
for any positive integer $m$, in the last formula assuming that $m<n$. Therefore
$$(n+k_0)! (n-k_0-2)! > \left( \frac {n+k_0}{e} \right)^{n+k_0} \left( \frac {n-k_0-2}{e} \right)^{n-k_0-2} =$$
$$ \left( \frac ne\right)^{2n-2}  \left( 1 + \frac {k_0}{n} \right)^{n+k_0} \left( 1 - \frac{k_0+2}{n} \right)^{n-k_0-2} > \left( \frac ne\right)^{2n-2} e^{-2}.$$
We conclude that
$$\left| \underset{k\neq k_0}{\prod_{k=1}^{n-1}} \left( A - 2\sqrt{B} \cos \frac{k\pi}{n} \right) \right| > \frac {1}{32e^4n^2} \left( \frac {4\sqrt{B}}{e^2} \right)^{n-2}.$$

If $|F_n(A,B)| \leq N$, then, via
$$ \left| \left(A - 2\sqrt{B} \cos \frac{k_0\pi}{n}\right)\right| < N \cdot 32e^4n^2 \left( \frac {e^2}{4 \sqrt{B}} \right)^{n-2} = 2^{9-2n} e^{2n} n^{2} NB^{-\frac12 n+1},$$
we obtain, by putting $c_3 =c_3(n) = 2^{9-2n} e^{2n} n^{2} $,
$$A \in \left(2\sqrt{B} \cos \frac{k_0\pi}{n} - c_3NB^{-\frac 12 n+1}, 2\sqrt{B} \cos \frac{k_0\pi}{n} + c_3N B^{-\frac 12 n+1}\right). $$
Therefore, for every $B$ the number of possible integers $A$ is at most
$$ 2c_3NB^{-\frac 12 n+1}.$$
Thus, since $n \geq 5$, in this case the number of possible pairs $(A,B)$ with $|F_n(A,B)| \leq N$ is at most
$$ 2c_3N \sum_{B\geq 17N^{\frac{2}{n-1}}} B^{-\frac 12 n+1} \leq 2c_3N \int_{17N^{\frac{2}{n-1}}-1}^{\infty} t^{-\frac 12 n+1}dt = $$
$$ \frac{4c_3}{n-4}  N (17N^{\frac {2}{n-1}}-1)^{-\frac 12 n+2} <  4c_3 N (16N^{\frac {2}{n-1}})^{-\frac 12 n+2} = 2^{19} n^2 \left( \frac {e^2}{16}\right)^n N^{\frac {3}{n-1}}.$$

\vskip.1cm
We conclude that there are at most $2 \cdot (306 + 2^{19-n} n^2) N^{\frac {3}{n-1}}$ pairs $(A,B)$ with $|F_n(A,B)| \leq N$.

\begin{proof}[Proof of the upper bound for $|\ml_{\geq n}(N)|$]
Let $m\ge n$ and suppose that $|U_m(A,B)| \leq N$. 
Observe that for an arbitrary Lucas sequence $U$, if $|U_m|\leq N$, then, by Theorem \ref{thm21}, Remark 1 and $\log(\sqrt{2}) > \frac 13$,
$$
m\leq  3\log N
$$
for $N$ sufficiently large. 
Observe further that 
$ 2^{20-(m+1)} (m+1)^2 < \frac 34\cdot 2^{20-m}m^2$ for $m \geq 5$ and that $N^{\frac {3}{n-1}}$ is monotonically decreasing in $n \geq 5$.
Therefore
$$|\ml_{\geq n}(N)| \leq \ml_n(N) + \sum_{m=n+1}^{\lfloor {3\log N} \rfloor} \left( 612 + 2^{20-n}n^2 \left(\frac 34\right)^{m-n} \right)N^{\frac {3}{n}} \leq$$
$$(612 +2^{22-n}n^2)N^{\frac {3}{n-1}}+ 1836N^{\frac {3}{n} }\log N = (612 +2^{22-n}n^2+o_n(1))N^{\frac {3}{n-1}}.$$

\end{proof}

\section{Proof of the lower bounds in Theorems \ref{thm23} and \ref{thm24}} \label{lower}

We consider $n$ to be fixed. By that also the polynomials $F_n(x,y)$ and
$$
G_m(x,y):=
\begin{cases}
F_n(\sqrt{x},y)& \text{for $n$ even},\\
F_n(\sqrt{x},y)/\sqrt{x}& \text{for $n$ odd}
\end{cases}
$$
are fixed. Note that $G_m$ is a binary form of degree $m= \frac{n-1}{2}$ if $n$ is odd and of degree $m=\frac{n-2}{2}$ if $n$ is even.
In our arguments we may clearly assume that $N$ is large enough.
If we write `for sufficiently large $N$' it means that the lower bound for $N$ may depend on $n$ only. 

For the proof of Theorem \ref{thm23} we follow a paper by Erd\H{o}s and Mahler \cite{em} who proved a similar result for binary forms.

Write $F_n(x,y) = \sum_{h=0}^{\lfloor {\frac{n-1}{2}} \rfloor} a_h x^{n-2h-1}y^h$. Then $a_0=1$. The discriminant $d$ of $F_n(x,1)$ is non-zero \cite{fhr}. Let $\gamma = \max(|d|,n)$, $R$ a non-zero integer and $\theta$ a number satisfying $0 < \theta < 1$, to be fixed later. Let $h(R)$ be the arithmetical function defined by
$$h(R) = \prod_{\gamma < p \leq N^{\theta},~ p^a ||R,~ p^a \leq N^{\theta}} ~p^a.$$

\begin{lemma} \label{lem51}
$$H(N) := \prod_{|x|\leq N,~ |y|\leq N^2,~ F_n(x,y) \neq 0} ~h(F_n(x,y)) \leq N^{24 \theta n N^3}.$$
\end{lemma}

\begin{proof} 
Since $d \neq 0$ and $a_0=1$, for given prime $p>\gamma$ and integers $a \geq 0$ and $y$, there are at most $n-1$ incongruent values of $x\pmod {p^a}$ for which $F_n(x,y) \equiv 0\pmod {p^a}$.
Therefore, for given $p$ and $a$ with $ \gamma < p \leq p^a \leq N^{\theta}$,
the conditions
$$|x|\leq N,~ |y|\leq N^2,~F_n(x,y) \neq 0, ~ F_n(x,y) \equiv 0\pmod {p^a}$$
have less than 
$$n(2N^2+1) \Bigl \lceil\frac{2N+1}{p^a} \Bigr \rceil  \leq \frac{6nN^3}{p^a}$$
solutions $x,y$. It follows that the exponent $b$ with $p^b ||H(N)$ satisfies the inequality
$$b \leq \sum_{a=1}^{\infty} \frac{6nN^3}{p^a} = \frac{6nN^3}{p-1} \leq \frac{12nN^3}{p}.$$
Hence, for sufficiently large $N$,
$$H(N) \leq \exp \left( \sum_{\gamma < p \leq N^{\theta}} ~\frac{12nN^3}{p} \log p\right) \leq N^{24\theta nN^3},$$
since
$$\sum_{p \leq u} \frac {\log p}{p} \leq 2 \log u$$
for sufficiently large $u$.
\end{proof}

\begin{lemma} \label{lem52}
If $\mu$ is the number of pairs $x,y$ with
$$|F_n(x,y)| \leq \sqrt{N}, |x| \leq N, |y| \leq N^2$$
then $\mu \leq N^3$ for sufficiently large $N$.
\end{lemma}

\begin{proof} For a given $m$ with $|m| \leq \sqrt{N}$ and a given $y$ with $|y| \leq N^2$, the equation $F_n(x,y) = m$ has at most $n-1$ integer solutions $x$ and therefore 
$$ \mu \leq  (n-1) (2\sqrt{N} + 1)(2N^2+1) \leq  N^3$$
for $N$ sufficiently large.
\end{proof}

\begin{lemma}\label{lem53}
For sufficiently large $N$ there are at least $2N^3$ pairs of integers $x,y$ with $|x| \leq N, |y| \leq N^2$ and $(x,y)=1$.
\end{lemma}
\begin{proof}
Obviously the number of these pairs is at least $4Q$, where $Q$ denotes the number of pairs with $1 \leq x \leq N, 1 \leq y \leq N^2, (x,y)=1$. Using that $\sum_p p^{-2} = 0.4522472...$ (see e.g. \cite {oeis} A085548) we have
$$Q \geq N^3 - \sum_p \frac {N^3}{p^2}  >   \frac {N^3}{2}.$$
\end{proof}

\begin{lemma} \label{lem54}
For sufficiently large $N$ there are at least $\frac 12 N^3$ pairs of integers $(x,y)$ with
$$|x| \leq N, ~|y| \leq N^2,~ F_n(x,y) \neq 0,~ (x,y)=1,~ h(F_n(x,y)) \leq |F_n(x,y)|^{96 \theta n}.$$
\end{lemma}

\begin{proof}
 By Lemmas \ref{lem52} and \ref{lem53} there are at least $N^3$ pairs $(x,y)$ with
 
 $$|x| \leq N, ~|y|\ \leq N^2, ~(x,y)=1,~|F_n(x,y)| \geq \sqrt{N}.$$
 Hence, if Lemma \ref{lem54} were false, there would be more than $\frac 12 N^3$ such coprime pairs $x,y$ with
 $$h\left(F_n(x,y)\right) > |F_n(x,y)|^{96 \theta n} \geq N^{48 \theta n}$$
 and therefore, for $N$ sufficiently large,
 $$H(N)> N^{48 \theta n \cdot \frac12 N^3} = N^{24 \theta n N^3},$$
 in contradiction to Lemma \ref{lem51}.
\end{proof}

\begin{lemma} \label{lem55}
For all $x$ and $y$,
$$ |F_n(x,y)| \leq F^*_{n} \cdot (\max(x^2,|y|))^{\frac 12 (n-1)} < (3 \max(x^2,|y|))^{\frac 12 (n-1)} ,$$
where $F^*_n$ is the $n$-th Fibonacci number.
\end{lemma}

\begin{proof}
By induction on $n$ it follows that $\sum_{h=0}^{n-1} |a_h| = F^*_n.$
\end{proof}

\begin{lemma} \label{lem56}
For sufficiently large $N$ there are at least $\frac 12 N^3$ pairs of integers $x,y$ with
\begin{equation} \label{formone}
 |x| \leq N, ~|y| \leq N^2,~ (x,y)=1,~ F_n(x,y) \neq 0
 \end{equation}
such that $|F_n(x,y)| = k_1k_2$ where $k_1$ and $k_2$ are positive integers such that $k_1$ is divisible by at most $\gamma + \frac{n \log (2N)}{\theta \log N}$ different primes and 
$k_2 \leq |F_n(x,y)|^{96 \theta n}.$
\end{lemma}

\begin{proof}
We apply Lemma \ref{lem54} with 
$$k_2 = h(F_n(x,y)),~~k_1 = \frac{|F_n(x,y)|}{k_2}.$$
Then $k_1$ and $k_2$ are positive integers, since $h(F_n(x,y))$ is a positive integer which divides $F_n(x,y)$. By Lemma \ref{lem54}, for at least $\frac 12 N^3$ pairs $(x,y)$ satisfying \eqref{formone},
$$k_2 \leq |F_n(x,y)|^{96 \theta n}.$$
The factor $k_1$ is divisible by prime numbers of the form $p$ with $p \leq \gamma$ or $p > N^{\theta}.$ But 
since $|F_n(x,y)| < F^*_nN^n< (2N)^n$ by Lemma \ref{lem55}, there are at most 
$$\gamma + \frac{n \log (2N)}{\theta \log N}$$ primes dividing $k_1$.
\end{proof}

It now suffices to prove that the equation $F_n(x,y) = k$ has a number of relatively prime integers solutions $(x,y)$ which is bounded by a constant depending only on $n$. Here we apply Lemma \ref{lem57} which is Theorem 3(i) of Lewis and Mahler \cite{lm}. Recall that $c_4, c_5, c_6, \dots$ are constants depending only on $n$. For an integer $x$ and a prime $p$ we write $|x|_p = p^{-r}$ if $p^r\mid x$, $p^{r+1} \nmid x$.

\begin{lemma} \label{lem57}
Let $G(x,y)$ be a binary form of degree $m \geq 3$ with integer coefficients and non-zero discriminant satisfying
$$G(1,0) \neq 0 \ \ and \ \ G(0,1) \neq 0.$$
Let $a$ be the canonical height of $G(x,y)$; let
$$ \rho_m = \min_{h=1, 2, \dots, m} \left( \frac{m}{h+1} + h\right),\ \  \sigma_m = \rho_m + \frac 1m;$$
and let $p_1,p_2, \dots, p_t$ be any finite number of primes.
There are not more than 
\begin{equation} \label{embound}
2^{ \frac{\sigma_m + 2}{ \sigma_m - 2}} (2m^2a)^{\frac{4m}{\sigma_m - 2}} +e(t+1)\Bigl \lfloor \frac{\log(48a^2m^8)}{ \log(a-1)} + 2\Bigr \rfloor (em(3\rho_m m+4))^{t+1}
\end{equation}
pairs of integers $x,y$ such that
$$ x \neq 0, \ \ y>0, \ \ (x,y)=1, \ \ 0 < |G(x,y)| \prod_{\tau = 1}^t |G(x,y)|_{p_{\tau}} \leq (\max(|x|,|y|))^{c_4},$$
where $c_4 = m-\rho_m-\frac{4}{3m}$ for all $m \geq 3$.
\end{lemma}

On  using that $\rho_m  \leq 2 \sqrt{m}$, we obtain
$$c_4 = m-\rho_m - \frac{4}{3m}= \frac {1} {18},~ \frac {2} {3},~  \frac {37} {30},~ \frac{16}{9},~ \frac{52}{21},~ \frac {19}{6},~ \frac{104}{27},~ \geq \frac m3  $$
for $m=3,4, 5, 6, 7, 8, 9, \geq 10$ respectively.
Recall that $n$ is odd, $n \geq 7$, hence $F_n(\sqrt{x},y) \equiv G_m(x,y)$ and $n-1=2m$. Thus $m \geq 3$ and \eqref{embound} can be written as $c_5 + c_6(t+1)c_7^{t+1}$,
provided that $a$ can be bounded by a constant depending only on $n$.
By Lemma \ref{lem55} we have $|a|<2^n$ for our polynomial $G(x,y)$.

\begin{lemma} \label{lem58} 
If $k$ is an integer, $k > (F_{2m+1}^*)^{5}$, and it can be written in the form $k=k_1k_2$ with $k_1$ and $k_2$ positive integers such that $k_1$ is divisible by at most $t$ different primes,  $k_2$ is not divisible by any of these primes and $k_2 \leq k^{\frac {2c_4}{5m}}$, then the equation $G_m(x,y) = k$ has not more than $c_5 + c_6(t+1)c_7^{t+1}$ different solutions in relatively prime integers $x,y$.
\end{lemma}

\begin{proof} 
Let $k$ be an integer for which $G_m(x,y)=k$ has at least one solution $(x_0,y_0)$. Then, by Lemma \ref{lem55},
\begin{equation} \label{k55}
 k  \leq F^*_{2m+1} \cdot (\max(|x|,|y|))^{2m}.
 \end{equation}
The integer $k=k_1k_2$ is the product of the integer $k_1$ with at most $t$ distinct prime divisors and the integer $k_2$ which is coprime to the product of these primes. By \eqref{k55},
$$k_2 \leq k^{\frac {2c_4}{5m}}  \leq (F^*_{2m+1})^{\frac {2c_4}{5m}}  (\max(|x|,|y|))^{\frac {4c_4} {5}}.$$
Thus, since $k > (F_{2m+1}^*)^{5}$,  we get, by \eqref{k55},
$$\max(|x|,|y|) \geq \left( \frac{k}{F^*_{2m+1}} \right)^{\frac {1}{2m}} > (F^*_{2m+1})^{\frac {2}{m}},$$
hence
$$k_2 <  (\max(|x|,|y|))^{\frac{c_4}{5}+ \frac{4c_4}{5}}=  (\max(|x|,|y|))^{c_4}.$$
The statement of Lemma \ref{lem58} follows from Lemma \ref{lem57}.
\end{proof}

\begin{proof}[Proof of Theorem \ref{thm23}.] 
Let $R = \left(\frac{N}{F^*_{n}}\right)^{\frac{1}{n-1}}$. Then $|F_n(x,y)| \leq N$ for $|x| \leq R$ and $|y| \leq R^2$ in view of Lemma \ref{lem55}. It follows from Lemma \ref{lem56} that there are at least
\begin{equation}  \label{NFn}
\frac 12 R^3 = \frac 12 \left(\frac{N}{F^*_{n}}\right)^{\frac {3}{n-1}}
\end{equation}
different pairs of relatively prime integers $x,y$ with $|x| \leq R, |y| \leq R^2$ for which 
$|F_n(x,y)| = k \neq 0$
and $k$ is a product of two positive integers $k=k_1k_2$, such that $k_1$ is divisible by at most $\gamma + \frac{n \log(2R)}{\theta \log R}$ different primes, whereas $k_2 \leq |F_n(x,y)|^{96 \theta n}$. We choose $$\theta = \frac{c_4}{240 n^2} ~( <1)$$
where $c_4$ is the constant in Lemma \ref{lem57} for $m = \frac{n-1}{2}$.
Since $n\geq 7$, $G_m(x,y)$ is a binary form of degree $m= \frac 12 (n-1) \geq 3$. 
Thus there are at least \eqref{NFn} different pairs of relatively prime integers $x,y$ with $|x| \leq R, |y| \leq R^2$ for which 
$|G_m(x^2,y)| =|F_n(x,y)| = k \neq 0$
and  $k=k_1k_2$, such that $k_1$ is divisible by at most $$\gamma + \frac{n \log(2R)}{\theta \log R} < c_8n^3$$ different primes for some positive integer $c_8$, while 
$$k_2 \leq |G_m(x^2,y)|^{96 \theta n} = |G_m(x^2,y)|^{\frac{2c_4}{5n}}.$$ Here we have used that $N$, hence $R$, is sufficiently large.
We apply Lemma \ref{lem58} with $t=c_8 n^3$ and $k_2 \leq k^{\frac{2c_4}{5n}}$.  We conclude that the number of different relatively prime solutions of $G_m(x,y) = k$ with $k > (F^*_{n})^5$  is not larger than $$ c_9:=c_5 + c_6(c_8 n^3)c_7^{c_8 n^3}.$$ So the number of different coprime solutions of $F_n(x,y) = G_m(x^2,y)=k$ for a fixed $k$ is not larger than $c_{9}.$ The number of such solutions of the equation $F_n(x,y) =k$ with $k \leq (F_{n}^*)^5$ is bounded by a constant $c_{10}$. Put $c_{11} = c_9+c_{10}$. Thus, for sufficiently large $N$, by \eqref{NFn}, there must be at least
$$ \frac {1}{2c_{11}} \left(\frac{N}{F^*_{n}}\right)^{\frac {3}{n-1}} > c_{12} N^{\frac{3}{n-1}}+9$$
different positive integers $k \leq N$ for which $|F_n(x,y)| = k$ has at least one solution in integers $x,y$ with $(x,y)=1.$ Observe that in view of Lemma \ref{lemdeg} for these values of $x,y$, the Lucas sequence corresponding to $(A,B)=(x,y)$ is degenerate only for the nine values
$$
(x,y)=(0,0),\ (0,\pm 1),\ (\pm 1,0),\ (\pm 1,\pm 1).
$$
This has been proved for $N > N_1(n)$. By adjusting the constant $c_{12}$ for the smaller values of $N$ and calling the new constant $c_1$, the result follows.
\end{proof}

\noindent {\bf Remark 5.} Obviously the same lower bound applies to $\ml_{\geq n}$.

\begin{proof}[Proof of Theorem \ref{thm24} - case $n$ even] Let $n$ be even. As we already mentioned, in this case $U_n=AG_m(A^2,B)$, where $A$ divides $U_n$, and $G_m(x,y)$ is a binary form of degree $m=\frac{n-2}{2}$. By Lemma \ref{lem55} for all the pairs $(A,B)$ with
$$
1\leq A\leq \frac 12 N^{\frac{1}{n-1}},\ \ \ 1\leq B\leq  \frac 14 N^{\frac{2}{n-1}}
$$
we have $|G_m(A,B)|\leq N$. Observe that the number of such pairs $(A,B)$ for which the corresponding Lucas sequence is non-degenerate, by Lemma \ref{lemdeg} is at least $\frac 18 N^{\frac{3}{n-1}}-2N^{\frac{1}{n-1}}$. Indeed, for $A$ fixed, there exists at most one such $B$ with $A^2-4B=0$, and at most three such $B$-s with \eqref{eqdeg}. We claim that for any $k$ with $1\leq k\leq N$ the equation
\begin{equation}
\label{hany}
F_n(x,y) = AG_m(x^2,y)=k
\end{equation}
can have at most $nN^{\frac{1.066}{\log\log N}}$ solutions among these pairs $(A,B)$. From this the theorem immediately follows. In the first place, observe that if $(A,B)$ is any solution of \eqref{hany}, then $A\mid k$. By Theor\`eme 1 of \cite{nr} the number of divisors of $k$, and so the number of choices of $A$, is at most
\begin{equation}
\label{dn}
\tau(k)\leq k^{\frac{1.5379\log 2}{\log\log k}}\leq N^{\frac{1.066}{\log\log N}},
\end{equation}
where $\tau(k)$ denotes the number of positive divisors of $k$.
Furthermore, for any $x=A$, \eqref{hany} can have at most $n$ solutions in $y$, and our claim follows.

Obviously the same lower bound applies to $\ml_{\geq n}$. Hence the statement is proved for even $n$.
\end{proof}

Theorem \ref{thm24} is a consequence of Theorem \ref{thm23} for $n\geq 7$ odd, while it has been proved above for $n\geq 6$ even. In order to deal with the remaining case $n=5$ we apply the following lemma.

\begin{lemma}
\label{lem5}
There exists an absolute constant $T_0$, such that if $T$ is an integer with $T>T_0$ and $t$ is a non-zero integer with $|t|\leq 4T$, then the equation
\begin{equation}
\label{pell5}
x^2-5y^2=t
\end{equation}
has at most $T^{\frac{4}{\log\log T}}$ solutions in positive integers $x,y$ with $|y|\leq \frac{1}{2}\sqrt{T}$.
\end{lemma}

\begin{proof} It is standard (see e.g. Lemma 3.2 and its proof in \cite{hs}) that all positive solutions of \eqref{pell5} come from identities
$$
x+\sqrt{5}y=(u_i+v_i\sqrt{5})(9+4\sqrt{5})^n\ \ \ (n\geq 0,\ i=1,\dots,I),
$$
where $u_i,v_i\in\mathbb{Z}$ $(i=1,\dots,I)$ for some $I$. Here $9+4\sqrt{5}$ is the fundamental solution of \eqref{pell5} with $t=1$. By Theorem 8-9 of \cite{lvq} (see pp. 147-148) we can take
\begin{equation}
\label{u}
0<u_i<\sqrt{\frac{(22+9\sqrt{5})|t|}{8}}<4.6\sqrt{T}\ \ \ (i=1,\dots,I),
\end{equation}
which implies
\begin{equation}
\label{v}
|v_i|<2.3\sqrt{T}\ \ \ (i=1,\dots,I).
\end{equation}
Further, by Lemma 5 of Gy\H{o}ry \cite{gy} we know that
$$
I\leq 2^{\omega(t)}\tau_2(t),
$$
where $\omega(t)$ is the number of prime divisors of $t$, while $\tau_2(t)$ is the number of ways the principal ideal $(t)$ can be factorized into the product of two ideals in $\mathbb{Q}(\sqrt{5})$. As every prime can split into the product of at most two prime ideals in $\mathbb{Q}(\sqrt{5})$, we have $\tau_2(t)\leq(\tau(t))^2$. Thus, in view of \eqref{dn} we have
\begin{equation}
\label{i}
I\leq \tau(t)^3\leq |t|^{\frac{3.2}{\log\log |t|}}\leq (4T)^{\frac{3.2}{\log\log (4T)}}.
\end{equation}
Using Lemma 3.2 from \cite{hs} and its proof, we get that $y$ belongs to a recurrence sequence $G^{(i)}$ $(i \in \{1,\dots,I\})$ with initial terms
$$
(G_0^{(i)},G_1^{(i)})=(v_i,\mu_2u_i+2\mu_1v_i)=(v_i,4u_i+18v_i)
$$
and recurrence relation
$$
G_{n+2}^{(i)}=2\mu_1G_{n+1}^{(i)}-G_n^{(i)}=18G_{n+1}^{(i)}-G_n^{(i)}\ \ \ (n\geq 0),
$$
because $\mu_1+\mu_2\sqrt{5}=9+4\sqrt{5}$. Setting
$$
\alpha=9+4\sqrt{5}\ \ \ \text{and}\ \ \ \beta=9-4\sqrt{5},
$$
we have
\begin{equation}
\label{y}
y=a_i\alpha^n-b_i\beta^n
\end{equation}
for some $n\geq 0$ and $1\leq i\leq I$ with
$$
a_i=\frac{4u_i+(9+4\sqrt{5})v_i}{8\sqrt{5}}\ \ \ \text{and}\ \ \ b_i=\frac{4u_i+(9-4\sqrt{5})v_i}{8\sqrt{5}}.
$$
By \eqref{u} and \eqref{v} we have
\begin{equation}
\label{b}
|b_i|<1.1\sqrt{T}.
\end{equation}
On the other hand, letting
$$
\gamma=\frac{1}{4u_i+(9+4\sqrt{5})v_i}
$$
we have
$$
(16u_i^2+72u_iv_i+v_i^2)\gamma^2-(8u_i+18v_i)\gamma+1=0.
$$
Clearly, $16u_i^2+72u_iv_i+v_i^2=(4u_i+9v_i)^2-80v_i^2\neq 0$. Thus letting
$$
h:=\max(|8u_i+18v_i|,1),
$$
we see that either $|\gamma|\leq 1$, or
$$
|\gamma^2|\leq |16u_i^2+72u_iv_i+v_i^2|\gamma^2\leq h(|\gamma|+1)\leq\frac{h|\gamma^2|}{|\gamma|-1}
$$
whence $|\gamma|\leq h+1$. From this and \eqref{u} and \eqref{v} we get that
$$
|\gamma|<79\ T.
$$
This implies, by $a_i \gamma = \frac {1}{8\sqrt{5}}$,
$$
|a_i|>\frac{1}{1414\ T}.
$$
Combining this with \eqref{y} and \eqref{b} we obtain
$$
y\geq |a_i\alpha^n|-|b_i\beta^n|>\frac{1}{1414T}~(17.9)^n - \sqrt T,
$$
and this exceeds $\frac 12 \sqrt{T}$ if $n>\frac{1.5 \log T + \log 1414}{\log 17.9}$. This implies that $y>\frac{1}{2}\sqrt{T}$ for $n> \log T + 3$. So in view of \eqref{i}, our claim follows by a simple calculation.
\end{proof}

\begin{proof}[Proof of Theorem \ref{thm24} - case $n=5$]
Recall $F_5(A,B)=A^4-3A^2B+B^2$, and for a given positive integer $N$ consider the set
$$
H_N:=\left\{(A,B)\in{\mathbb Z}^2\ :\ 1\leq A\leq \frac{1}{2}N^{\frac{1}{4}},\ 1\leq B\leq \frac{1}{2}N^{\frac{1}{2}}\right\}.
$$
Then for $(A,B)\in H_N$ we have that the corresponding Lucas-sequence (by Lemma \ref{lemdeg}, similarly as in case of $n$ odd) is degenerate at most for $2N^{\frac{1}{4}}$ pairs $(A,B)$, $|F_5(A,B)|<N$ and also that $|H_N|=\frac{1}{4} N^{\frac{3}{4}}$.

We show that for any $k$ with $|k|\leq N$, there are only `few' $(A,B)\in H_N$ such that
\begin{equation}
\label{f5}
F_5(A,B)=k.
\end{equation}
As $4F_5(A,B)=(2B-3A^2)^2-5A^4$, \eqref{f5} can be written as
$$
x^2-5y^2=t
$$
with $x=2B-3A^2$, $y=A^2$ and $t=4k$. So applying Lemma \ref{lem5} with $T=N$, we see that \eqref{f5} has at most $N^{\frac{4}{\log\log N}}$ solutions with $N>N_2(n)$. Hence for $N>N_2(n)$, we have
$$
|{\mathcal L}_5(N)|\geq \frac 15 N^{\frac{3}{4}-\frac{4}{\log\log N}},
$$
and the statement follows.
\end{proof}

\section*{Acknowledgement} We thank Jan-Hendrik Evertse for providing useful references to us.

\end{document}